\documentclass{article}
\usepackage{amsmath, amssymb, amsthm, mathtools}
\usepackage{hyperref}
\usepackage{authblk}
\usepackage[margin=1in]{geometry}
\usepackage{todonotes}
\usepackage{cleveref}
\usepackage{subcaption}
\usepackage{cases}
\usepackage{tikz}
\usepackage{multirow}
\usetikzlibrary{calc}
\usetikzlibrary{decorations.pathreplacing}

\newtheorem{theorem}{Theorem}[section]
\numberwithin{theorem}{section}

\newtheorem{lemma}[theorem]{Lemma}

\newtheorem{proposition}[theorem]{Proposition}
\newtheorem{definition}[theorem]{Definition}
\theoremstyle{definition}
\newtheorem{example}{Example}

\newcommand\pihat{\hat{\pi}}
\newcommand\taubar{\overline\tau}
\newcommand{\Z}{\mathbb{Z}}
\newcommand{\shift}{\operatorname{shift}}
\newcommand{\x}{\mathbf x}
\newcommand{\xtau}{(\x,\tau)}
\newcommand{\deltaxtau}{\left(\delta,\xtau\right)}
\newcommand{\shiftBlock}[1]{\shift_{s_{#1}}\!\!\left(\block(\hat\delta, \hat\tau_{#1})\right)}
\newcommand{\Cyc}{\operatorname{Cyc}}
\newcommand{\Fxpt}{\operatorname{Fxpt}}
\newcommand{\block}{\operatorname{block}}

\begin{document}
\title{A statistic-swapping involution on the Cartesian product of the symmetric group \texorpdfstring{$S_{kn}$}{on kn letters} and the generalized symmetric group \texorpdfstring{$S(k,n)$}{S(k,n)}}
\author{Peter Kagey}
\affil{Cal Poly Pomona}
\author{Kai Mawhinney}
\affil{Harvey Mudd College}
\date{\today}
\maketitle

\begin{abstract}
  We construct a statistic-swapping involution on the Cartesian product of the generalized symmetric group $S(k,n)$ with the symmetric group $S_{kn}$, which swaps the number of fixed points in the generalized symmetric group element with the number of $k$-cycles in the symmetric group element. This gives a combinatorial proof for a probabilistic observation: the distribution of fixed points on $S(k,n)$ matches the distribution of $k$-cycles on $S_{kn}$.
\end{abstract}

\section{Introduction}

This paper gives a combinatorial proof of a statement in probability that shows that two statistics on two distinct combinatorial objects have the same distribution under a uniform sampling of the respective objects. If we choose a permutation $\pi\in S_{kn}$ and element of the  generalized symmetric group $\sigma \in S(k,n) = \mathbb{Z}_k \wr S_n = \mathbb{Z}_k^n \rtimes S_n$ uniformly at random, then
the probability $\pi$ has exactly $m$ $k$-cycles is equal to the probability that $\sigma $ has exactly $m$ \emph{fixed points}.

There are many notable existing results on \emph{derangements} of the generalized symmetric group $S(k,n)$, which are elements with no \emph{fixed points}. Gordon and McMahon \cite{GordonMcMahon2010} enumerate type-$B$ facet derangements, which occur when $k=2$ and thus $S(2,n) \cong B_n$, the hyperoctahedral group. Extending their work, Assaf \cite{Assaf2010} enumerates derangements in $S(k,n)$ for arbitrary $k$ (which appears in the On-Line Encyclopedia of Integer Sequences as A320032 \cite{OEIS}) and provides $q$- and $(q,t)$-analogs. A different approach for analyzing derangements of $S(k,n)$ using Euler's difference table can be found in Faliharimalala and Zeng \cite{FaliharimalalaZeng2008}.

Chow \cite{Chow2005} studied $q$-enumeration of derangements in the hyperoctahedral group $B_n$ via the \textit{flag major index} to obtain a $q$-analogue of
the type-$B$ derangement number, and later did something similar in type-$D$ \cite{Chow2023}, analyzing the number of $n$-derangements by their major indices, and demonstrated the ratio monotonicity of the coefficients of their generating function. We will consider type-$B$ derangements only as a special case, where $k=2$ and the element has no fixed points. We begin this paper with some preliminary definitions.

\begin{definition}
    A \emph{fixed point} of an element $\sigma = \xtau \in \mathbb{Z}_k^n \rtimes S_n$ of the generalized symmetric group is an index $1 \leq i \leq n$ such that both $x_i = 0 \in \Z_k$ and $\tau(i) = i$.
\end{definition}
\begin{example}
    Suppose $k = 3$, $n=5$, and $\sigma = \xtau = \left((0,1,0,2,1),(2)(3)(514)\right) \in \Z_3^5 \rtimes S_5$. Then $\sigma$ has a single fixed point, which occurs at index $3$ since $x_3 = 0$ and $\tau(3) = 3$.
    \label{ex:fixedPoint}
\end{example}

In the case $k=1$, both $S(k,n)$ and $S_{kn}$ are isomorphic to $S_n$. Fixed points and $1$-cycles are equivalent in this setting, so the map $(\sigma,\pi) \mapsto (\pi,\sigma)$ trivially swaps relevant statistics. In the case $k=2$, $S(2,n)$ is the hyperoctahedral group $B_n$: the symmetries of the $n$-dimensional hypercube. This can be seen by interpreting $S(2,n)$ as $\mathbb{Z}_2^n\rtimes S_n$, where the element $\pi\in S_n$ is a permutation of the diagonals of the cross-polytope, and $\sigma\in \mathbb{Z}_2^n$ determines whether the points on this diagonal stay in place (the identity), or swap (a transposition). With this interpretation, a fixed point of $\sigma \in S(2,n)$ is a pair of opposite vertices of the cross polytope which $\sigma$ fixes.

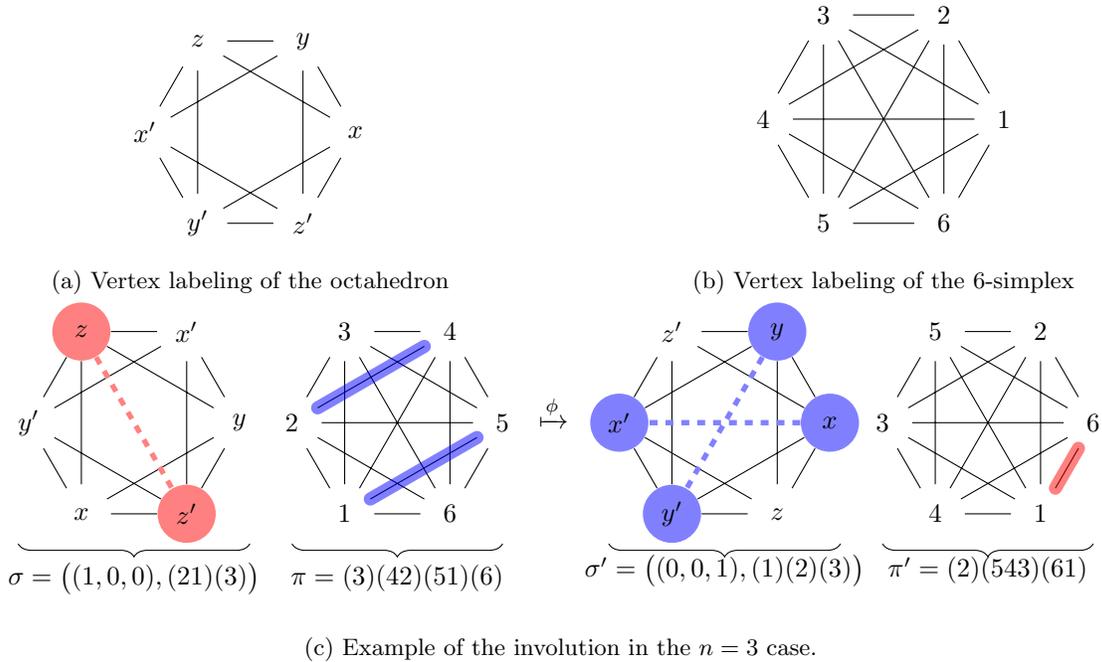
\begin{figure}
\begin{subfigure}[b]{0.49\textwidth}
\centering
\begin{tikzpicture}[scale=1.4,baseline=-2]
    \node[circle, minimum size=2.2em] (X1) at ({cos 0}, {sin 0}) {$x$};
    \node[circle, minimum size=2.2em] (Y1) at ({cos 60}, {sin 60}) {$y$};
    \node[circle, minimum size=2.2em] (Z1) at ({cos 120}, {sin 120}) {$z$};
    \node[circle, minimum size=2.2em] (X2) at ({cos 180}, {sin 180}) {$x'$};
    \node[circle, minimum size=2.2em] (Y2) at ({cos 240}, {sin 240}) {$y'$};
    \node[circle, minimum size=2.2em] (Z2) at ({cos 300}, {sin 300}) {$z'$};

    \draw
        (X1)--(Y1)
        (X1)--(Z1)
        (X1)--(Y2)
        (X1)--(Z2)
        (Y1)--(Z1)
        (Y1)--(X2)
        (Y1)--(Z2)
        (Z1)--(X2)
        (Z1)--(Y2)
        (X2)--(Y2)
        (X2)--(Z2)
        (Y2)--(Z2)
    ;
\end{tikzpicture}
\caption{Vertex labeling of the octahedron}
\label{subfig:cubeDefault}
\end{subfigure}
\begin{subfigure}[b]{0.49\textwidth}
\centering
\begin{tikzpicture}[scale=1.6]
        \node[circle, minimum size=2.2em] (A1) at ({cos 0}, {sin 0}) {1};
        \node[circle, minimum size=2.2em] (A2) at ({cos 60}, {sin 60}) {2};
        \node[circle, minimum size=2.2em] (A3) at ({cos 120}, {sin 120}) {3};
        \node[circle, minimum size=2.2em] (A4) at ({cos 180}, {sin 180}) {4};
        \node[circle, minimum size=2.2em] (A5) at ({cos 240}, {sin 240}) {5};
        \node[circle, minimum size=2.2em] (A6) at ({cos 300}, {sin 300}) {6};
    \draw
        (A1) -- (A2)
        (A1) -- (A3)
        (A1) -- (A4)
        (A1) -- (A5)
        (A1) -- (A6)
        (A2) -- (A3)
        (A2) -- (A4)
        (A2) -- (A5)
        (A2) -- (A6)
        (A3) -- (A4)
        (A3) -- (A5)
        (A3) -- (A6)
        (A4) -- (A5)
        (A4) -- (A6)
        (A5) -- (A6)
    ;
\end{tikzpicture}
\caption{Vertex labeling of the $6$-simplex}
\label{subfig:simplexDefault}
\end{subfigure}
\begin{subfigure}[b]{0.99\textwidth}
\[
\begin{tikzpicture}[scale=1.4,baseline=-2]
    \node[circle, minimum size=2.2em] (Z1) at ({cos 0}, {sin 0}) {$y$};
    \node[circle, minimum size=2.2em] (Y1) at ({cos 60}, {sin 60}) {$x'$};
    \node[circle, minimum size=2.2em, fill=red!50!white] (Z1) at ({cos 120}, {sin 120}) {$z$};
    \node[circle, minimum size=2.2em] (Z2) at ({cos 180}, {sin 180}) {$y'$};
    \node[circle, minimum size=2.2em] (Y2) at ({cos 240}, {sin 240}) {$x$};
    \node[circle, minimum size=2.2em, fill=red!50!white] (Z2) at ({cos 300}, {sin 300}) {$z'$};

    \draw
        (X1)--(Y1)
        (X1)--(Z1)
        (X1)--(Y2)
        (X1)--(Z2)
        (Y1)--(Z1)
        (Y1)--(X2)
        (Y1)--(Z2)
        (Z1)--(X2)
        (Z1)--(Y2)
        (X2)--(Y2)
        (X2)--(Z2)
        (Y2)--(Z2)
    ;

    \draw[dashed, red!50, line width=0.2em] (Z1)--(Z2);

    \begin{scope}[xshift=2.5cm]
        \node[circle, minimum size=2.2em] (A1) at ({cos 0}, {sin 0}) {5};
        \node[circle, minimum size=2.2em] (A2) at ({cos 60}, {sin 60}) {4};
        \node[circle, minimum size=2.2em] (A3) at ({cos 120}, {sin 120}) {3};
        \node[circle, minimum size=2.2em] (A4) at ({cos 180}, {sin 180}) {2};
        \node[circle, minimum size=2.2em] (A5) at ({cos 240}, {sin 240}) {1};
        \node[circle, minimum size=2.2em] (A6) at ({cos 300}, {sin 300}) {6};
    \end{scope}
    \draw
        (A1) -- (A2)
        (A1) -- (A3)
        (A1) -- (A4)
        (A1) -- (A5)
        (A1) -- (A6)
        (A2) -- (A3)
        (A2) -- (A4)
        (A2) -- (A5)
        (A2) -- (A6)
        (A3) -- (A4)
        (A3) -- (A5)
        (A3) -- (A6)
        (A4) -- (A5)
        (A4) -- (A6)
        (A5) -- (A6)
    ;
    \draw[line width=5, opacity=0.5, blue, line cap=round] (A2) -- (A4);
    \draw[line width=5, opacity=0.5, blue, line cap=round] (A1) -- (A5);

    \draw [decorate,decoration={brace,amplitude=5pt}]
  (1.1,-1.15) -- (-1.1,-1.15) node[midway,yshift=-1.3em]{$\sigma = \left((1,0,0), \operatorname{(21)(3)}\right)$};
    \draw [decorate,decoration={brace,amplitude=5pt}]
  (3.5,-1.15) -- (1.5,-1.15) node[midway,yshift=-1.3em]{$\pi = (3)(42)(51)(6)$};
\end{tikzpicture}
\xmapsto{\phi}
\begin{tikzpicture}[scale=1.4,baseline=-2]
    \node[circle, minimum size=2.2em, fill=blue!50!white] (X1) at ({cos 0}, {sin 0}) {$x$};
    \node[circle, minimum size=2.2em, fill=blue!50!white] (Y1) at ({cos 60}, {sin 60}) {$y$};
    \node[circle, minimum size=2.2em] (Z1) at ({cos 120}, {sin 120}) {$z'$};
    \node[circle, minimum size=2.2em, fill=blue!50!white] (X2) at ({cos 180}, {sin 180}) {$x'$};
    \node[circle, minimum size=2.2em, fill=blue!50!white] (Y2) at ({cos 240}, {sin 240}) {$y'$};
    \node[circle, minimum size=2.2em] (Z2) at ({cos 300}, {sin 300}) {$z$};

    \draw
        (X1)--(Y1)
        (X1)--(Z1)
        (X1)--(Y2)
        (X1)--(Z2)
        (Y1)--(Z1)
        (Y1)--(X2)
        (Y1)--(Z2)
        (Z1)--(X2)
        (Z1)--(Y2)
        (X2)--(Y2)
        (X2)--(Z2)
        (Y2)--(Z2)
    ;

    \draw[dashed, blue!50, line width=0.2em]
        (X1)--(X2)
        (Y1)--(Y2)
    ;

    \begin{scope}[xshift=2.5cm]
        \node[circle, minimum size=2.2em] (A1) at ({cos 0}, {sin 0}) {6};
        \node[circle, minimum size=2.2em] (A2) at ({cos 60}, {sin 60}) {2};
        \node[circle, minimum size=2.2em] (A3) at ({cos 120}, {sin 120}) {5};
        \node[circle, minimum size=2.2em] (A4) at ({cos 180}, {sin 180}) {3};
        \node[circle, minimum size=2.2em] (A5) at ({cos 240}, {sin 240}) {4};
        \node[circle, minimum size=2.2em] (A6) at ({cos 300}, {sin 300}) {1};
    \end{scope}
    \draw
        (A1) -- (A2)
        (A1) -- (A3)
        (A1) -- (A4)
        (A1) -- (A5)
        (A1) -- (A6)
        (A2) -- (A3)
        (A2) -- (A4)
        (A2) -- (A5)
        (A2) -- (A6)
        (A3) -- (A4)
        (A3) -- (A5)
        (A3) -- (A6)
        (A4) -- (A5)
        (A4) -- (A6)
        (A5) -- (A6)
    ;
    \draw[line width=5, opacity=0.5, red, line cap=round] (A6) -- (A1);

    \draw [decorate,decoration={brace,amplitude=5pt}]
  (1.1,-1.15) -- (-1.1,-1.15) node[midway,yshift=-1em]{$\sigma' = \left((0,0,1),(1)(2)(3)\right)$};
    \draw [decorate,decoration={brace,amplitude=5pt}]
  (3.5,-1.15) -- (1.5,-1.15) node[midway,yshift=-1em]{$\pi' = (2)(543)(61)$};
\end{tikzpicture}
\]
\caption{Example of the involution in the $n=3$ case.}
\label{subfig:map}
\end{subfigure}
\caption{In the $k=2$ and $n=3$ case, $S(2,3)=B_3$ acts on an octahedron, and $S_6$ acts on the $6$-simplex. Here, the fixed points of $S(2,3)$ fix pairs of antipodal vertices of the octahedron, and $k$-cycles reflect edges of the $6$-simplex. In this example, there is one fixed point and two $k$-cycles in the first pair, which map to a pair with two fixed points and one $k$-cycle.
}
\end{figure}

\begin{example}
    Let $k=2$ and $n=3$. Then there are $6! = 720$ permutations in $S_{kn}$ and $2^3 3! = 48$ elements of the generalized symmetric group $S(k,n)$. Their respective statistics, $k$-cycles and fixed points, are presented in \Cref{table:statistic}.
\begin{table}[h!]
    \centering
    \begin{tabular}{|l|r|r|r|r|r|}
      \hline
       & 0 & 1 & 2 & 3 & total \\
      \hline
      $S_{kn}$ & 435 & 225 & 45 & 15 & 720\\
      \hline
      $S(k,n)$ & 29 & 15 & 3 & 1 & 48\\
      \hline
    \end{tabular}
    \caption{The distribution of the number of $k$-cycles and the number fixed points in $S_{kn}$ and $S(k,n)$ respectively, where $k=2,n=3$. The value of the respective statistics is given by the top entry in each column.}
    \label{table:statistic}
\end{table}
\end{example}

\begin{theorem}
Let $\Cyc_m(S_{kn})$ denote the set of permutations in $S_{kn}$ with exactly $m$ $k$-cycles, and let $\Fxpt_m(S(k,n))$ denote the set of elements of $S(k,n)$ with exactly $m$ fixed points. Then \[
    \frac{
        \left|\Cyc_m(S_{kn})\right|
    }{
        \left|S_{kn}\right|
    } =
    \frac{
        \left|\Fxpt_m\!\left(S(k,n)\right)\right|
    }{
        \left|S(k,n)\right|
    }.
\]
\label{eq:mainTheorem}
\end{theorem}
This yields the following two combinatorial statements:
\begin{align}
    \left|\Cyc_m(S_{kn})\right|\left|S(k,n)\right|
    &=
    \left|\Fxpt_m\!\left(S(k,n)\right)\right|
    \left|S_{kn}\right|
    \\
    \left|\Fxpt_\alpha\!\left(S(k,n)\right)\right|
    \left|\Cyc_\beta(S_{kn})\right|
    &=
    \left|\Fxpt_\beta\!\left(S(k,n)\right)\right|
    \left|\Cyc_\alpha(S_{kn})\right|.
    \label{eq:combinatorialEquation}
\end{align}
We will provide a bijective proof of \Cref{eq:combinatorialEquation} (and thus \Cref{eq:mainTheorem}) by constructing an involution on $S_{kn} \times S(k,n)$ that swaps the relevant statistics: the number of $k$-cycles in the permutation is exchanged with the number of fixed points in the generalized symmetric group element.

\subsection{Decomposing the involution}
Instead of constructing the involution directly, we first describe and construct an auxiliary function.

\begin{definition}
    Throughout the paper, we will use the conventions that \begin{align*}
        \pi &\in S_{kn}, \\
        f(\pi) &= (\delta, \sigma) \in D_{k,n} \times S(k,n),\text{ and} \\
        \sigma &= \xtau \in S(k,n) = \mathbb{Z}_k^n \rtimes S_n,
    \end{align*} where
    $D_{k,n}$ denotes the set of permutations in $S_{kn}$ that consist of $n$ disjoint $k$-cycles, that is their cycle type is $(k,\dots,k) = (k^n)$:
    \[
        D_{k,n} = \left\{\pi\in S_{kn} \mid \operatorname{type}(\pi) = (k^n) \right\}.
    \]
\end{definition}

\begin{theorem}\label{thm:involution}
    Suppose that $f\colon S_{kn} \to D_{k,n} \times S(k,n)$ is a bijection such that $f(\pi) = \left(f_1(\pi), f_2(\pi)\right) = (\delta, \sigma)$ and the number of $k$-cycles in $\pi$ is equal to the number of fixed points in $\sigma$.

    Then the map $\phi \colon S(k,n) \times S_{kn} \to S(k,n) \times S_{kn}$ defined as \[
        \phi(\sigma', \pi) = \left(f_2(\pi),f^{-1}(f_1(\pi),\sigma')\right)
        = \left(\sigma, f^{-1}(\delta, \sigma')\right)
    \]
    is an involution, such that the number of $k$-cycles in $\pi$ is equal to the number of fixed points in $\sigma$ and the number of $k$-cycles in $\pi'$ is equal to the number of fixed points in $\sigma'$.
\end{theorem}
\begin{proof}
    Let $(\sigma', \pi) \in S(k,n) \times S_{kn}$. Denote $f(\pi) = (\delta, \sigma)$ and denote $f^{-1}(\delta, \sigma') = \pi'$.
    To check that this is an involution, we apply the map twice: \begin{align*}
        \phi(\phi(\sigma', \pi))
        &= \phi\!\left(\sigma, f^{-1}(\delta, \sigma')\right)
        = \phi(\sigma,\pi') \\
        &= \left(f_2(\pi'),f^{-1}(f_1(\pi'),\sigma)\right)
        = \left(\sigma',f^{-1}(\delta,\sigma)\right)
        = \left(\sigma', \pi\right)
    \end{align*}
    Moreover, by the assumption on $f$, if $\pi$ has $\alpha$ $k$-cycles, then $\sigma'$ has $\alpha$ fixed points, and if $\pi'$ has $\beta$ $k$-cycles, then $\sigma$ has $\beta$ fixed points. Thus $\phi(\sigma, \pi) = (\sigma', \pi')$ switches the number of $k$-cycles and fixed points.
\end{proof}
Therefore, to construct our desired involution, we construct $f\colon S_{kn} \to D_{k,n} \times S(k,n)$, which has the above properties. Our construction repeatedly uses Stanley's fundamental bijection and its inverse, which we define here.
\begin{definition}[Stanley's fundamental bijection \cite{Stanley2011EC1}]
We define the map $S_n \hat\to S_n$ by $\pi \mapsto \hat\pi$, where \[
    \pi =
    (c^{(1)}_1\cdots c^{(1)}_{\ell_1})
    (c^{(2)}_1\cdots c^{(2)}_{\ell_2})
    \cdots
    (c^{(t)}_1\cdots c^{(t)}_{\ell_t})
\] is written in canonical cycle notation (i.e., the first letter in each cycle, $c^{(i)}_1$, is the largest letter in its cycle, and the cycles are ordered in increasing order by first letter when read from left to right), and \[
    \hat\pi
    = \hat\pi_1\hat\pi_2\cdots\hat\pi_n
    = c^{(1)}_1\cdots c^{(1)}_{\ell_1}
    c^{(2)}_1\cdots c^{(2)}_{\ell_{2}}
    \cdots
    c^{(t)}_1\cdots c^{(t)}_{\ell_t}
\] is written as a word by dropping the parentheses in $\hat\pi$.

The map is inverted by placing a parenthesis at the start of the word, before each record read left to right, and at the end of the word.
\end{definition}

\subsection{The auxilliary bijection}
We will now define the auxiliary bijection $f \colon S_{kn} \to D_{k,n} \times S(k,n)$ describing each step with the aid of a running example. We start by defining the first coordinate of the output.
\begin{definition}
    The map $f_1\colon S_{kn} \to D_{k,n}$ is given by \[
        f_1(\pi) = \delta = (\hat\pi_1\hat\pi_2\cdots\hat\pi_k)
        (\hat\pi_{k+1}\hat\pi_{k+2}\cdots\hat\pi_{2k})
        \cdots
        (\hat\pi_{kn-k+1}\hat\pi_{kn-k+2}\cdots\hat\pi_{kn}),
    \]
    which is the result of applying Stanley's fundamental bijection to $\pi$, and then re-parenthesizing into $n$ disjoint $k$-cycles.
\end{definition}
Note that the resulting permutation is not necessarily written in canonical cycle notation.
\begin{example}
\label{ex:MainExample}
    This example is split into parts and will be continued throughout this section. Let $n=5$ and $k=3$, and let
    \[\pi = \mathrm{(8345)(9)(B1A)(F726CED)} \in S_{kn},\]
    where $1 < 2 < \dots < 9 < A < \dots < F$. Since $\pi$ is already written in canonical cycle notation, applying Stanley's bijection to $\pi$ means that $\hat\pi$ is the word that results from dropping the parentheses: \[
        \hat\pi = \mathrm{8345 9 B1A F726CED}
    \]
    Now, we write $\delta \in D_{k,n}$ in cycle notation by grouping every $k$ letters of $\hat\pi$ into a cycle: \[
        f_1(\pi) = \delta = \mathrm{(834)(59B)(1AF)(726)(CED)} = \mathrm{(726)(834)(B59)(EDC)(F1A)}.
    \]
\end{example}
Now we give a definition which will be useful in constructing the map $f_2(\pi) = (\mathbf{x},\tau) \in \mathbb{Z}_k^n \rtimes S_n$.
\begin{definition}
    Given a word $\hat\pi \in S_{kn}$, and $1 \leq i \leq n$, we denote the length-$k$ subword
    \[
        \block(\hat\pi,i) = \hat\pi_{ki-k+1} \hat\pi_{ki-k+2} \cdots \hat\pi_{ki},
    \]
    which we call the \em{$i$-th block} of $\hat\pi$.
\end{definition}
Now we define the first part of $f_2$.
\begin{definition}\label{def:tau}
    Given the word $\hat\pi\in S_{kn}$, define $\overline\tau = \overline\tau_1\dots\overline\tau_n$ where:
    \[
        \overline\tau_i = \max(\block(\hat\pi,i)).
    \]
    Let $\hat\tau$ be the word that is given by re-indexing $\overline\tau$ to letters in $[n]=\{1,2,\dots,n\}$, preserving the relative order.
    Let $\tau$ be the permutation resulting from applying the inverse of Stanley's fundamental bijection to $\hat\tau$.
\end{definition}
Notice that the blocks partition the word: $\hat\pi = \block(\hat\pi,1)\block(\hat\pi,2)\dots\block(\hat\pi,n)$.
\begin{example}
    Continuing with the previous example, with $\hat\pi = \mathrm{834\,5 9 B\,1A F\,726\,CED}$, we can compute $\taubar = \mathrm{8BF7E}$, $\hat\tau = 23514$, and (by applying the inverse of Stanley's fundamental bijection) $\tau = (2)(3)(514)$.
\end{example}
Next, we define the other part of $f_2$.

\begin{definition}\label{def:end-of-cycle-index}
    Let $\mathbf{x} = (x_1, x_2, \dots, x_n) \in \mathbb{Z}_k^n$ where $x_{i} \in \Z_k$ is the distance of the $i$th smallest letter of $\overline{\tau}$ to the end of its cycle in $\pi$, when $\pi$ is written in canonical cycle notation.

    Because $\overline{\tau}_i$ is the $\hat\tau_i$th smallest block leader, we may equivalently define $x_{\hat\tau_i}$ as the distance of $\overline{\tau}_i$ to the end of its cycle in $\pi$.
\end{definition}

The first perspective is simpler to write down $\mathbf{x}$, yet we will use the second more heavily throughout this paper.

\begin{example}
    Recall from the previous examples that $\taubar = \mathrm{8BF7E}$, so we may list the block leaders in ascending order as $78BEF$. These letters at distances of $6$, $4$, $3$, $2$, and $7$ respectively from the end of their cycles in $\pi = \mathrm{(8345)(9)(B1A)(F726CED)}$. We have \begin{align*}
        \x = (x_1, x_2, x_3, x_4, x_5) &\equiv (6,4,3,2,7) \equiv (0,1,0,2,1) \in \mathbb{Z}^5_3.
    \end{align*}
    We can instead use the second definition and see that they are equivalent. From left to right, the block leaders $\overline{\tau}_1,\dots,\overline{\tau}_5$ are $8$, $B$, $F$, $7$, and $3$, which are distances of $4$, $3$, $7$, $6$, and $2$ from the end of their cycles, respectively. Note $\hat\tau = 23514$. Thus:
    \begin{align*}
        \x = (x_1, x_2, x_3, x_4, x_5) &= (x_{\hat\tau_4}, x_{\hat\tau_1}, x_{\hat\tau_2}, x_{\hat\tau_5}, x_{\hat\tau_3})\equiv (6,4,3,2,7) \equiv (0,1,0,2,1) \in \mathbb{Z}^5_3.
    \end{align*}
\end{example}
Finally, we simply put the above components together.
\begin{definition}
    Let $f\colon S_{kn} \to D_{k,n} \times S(k,n)$ be defined as
    \(f(\pi) = \left(f_1(\pi),f_2(\pi)\right) = \deltaxtau\).
\end{definition}
We finish by checking that our example has the claimed property that the number of $k$-cycles in $\pi$ agrees with the number of fixed points in $\xtau$.
\begin{example}
     Continuing the above examples, $\pi = \mathrm{(8345)(9)(B1A)(F726CED)}$ has one $k$-cycle, $\mathrm{(B1A)}$. As shown in \Cref{ex:fixedPoint}, $\big(\underbrace{(0,1,0,2,1)}_\mathbf{x},\underbrace{(2)(3)(514)}_\tau\big)\in S(k,n)$ has one fixed point, namely $x_3=0$ and $\tau(3) = 3$.
\end{example}
In Section \ref{sec:StatisticsPreserved}, we will show that this construction always results in the number of $k$-cycles in the input agreeing with the number of fixed points in the image. In Section \ref{sec:Bijection}, we will show that we can invert $f$, and thus the map gives the desired bijection.

\section{\texorpdfstring{$k$}{k}-cycles map to fixed points}
\label{sec:StatisticsPreserved}
In this section, we prove that if $\pi$ has $m$ $k$-cycles and $f(\pi) = \left(\delta, \sigma'\right)$, then $\sigma' = \xtau \in S(k,n)$ has $m$ fixed points.

\begin{lemma}\label{lem:k-cycles}
    For $\pi \in S_n$ we have that $\pihat_i$ appears at the beginning of a $k$-cycle in $\pi$ if and only if
        $\pihat_j \le \pihat_i$ for all $j < i + k$ and
        $\pihat_{i+k} > \pihat_i$.
\end{lemma}
\begin{proof}
    First, assume $\pihat_i$ appears at the beginning of a $k$-cycle in $\pi$. Then it must be the case that $\pihat_{i+1},\ldots,\pihat_{i+k-1} < \pihat_i$ (otherwise the cycle containing $\pihat_i$ would have length $<k$) and that $\pihat_{i+k} > \pihat_i$ (otherwise the cycle containing $\pihat_i$ would have length $>k$). Moreover, since $\pihat_i$ appears at the beginning of a cycle, it must be the case that $\pihat_j < \pihat_i$ for all $j < i$.

    Conversely, if $\pihat_{i+k} > \pihat_i$ and $\pihat_j \le \pihat_i$ for all $j < i+k$, then it follows immediately from Stanley's fundamental bijection that $\pihat_i$ is at the beginning of a $k$-cycle in $\pi$.
\end{proof}

\begin{lemma}
    \label{lem:tauStartsKCycle}
    If $\left(\pihat_i \pihat_{i+1} \cdots \pihat_{i+k-1} \right)$ is a $k$-cycle of $\pi$, then $\pihat_i = \overline\tau_m$ for some $m \leq n$.
\end{lemma}
\begin{proof}
    By \Cref{lem:k-cycles}, we have $\pihat_j < \pihat_i$ for $1 \le j < i+k$ and $j \neq i$. In particular, $\pihat_i$ is the largest element in its block, so $\pihat_i = \overline\tau_m$ for some $m \le n$.
\end{proof}

\begin{theorem}\label{lem:fixed-points}
    For $\pi \in S_{kn}$ and $1 \le i \le n$, we have that $\taubar_i$ is at the beginning of a $k$-cycle in $\pi$ if and only if $\hat\tau_i$ is a fixed point of $f_2(\pi) = \xtau$.
\end{theorem}
\begin{proof}
    When applying the inverse of Stanley's fundamental bijection to $\hat\tau \in S_n$, $\hat\tau_i$ forms a $1$-cycle (and thus $\tau(\hat\tau_i) = \hat\tau_i$) if and only if $\hat\tau_i$ is a record when read left to right and either $i=n$ or $\hat\tau_i < \hat\tau_{i+1}$.

    First, assume that $\taubar_i$ is at the beginning of a $k$-cycle in $\pi$, and so $x_{\hat\tau_i} \equiv k\equiv 0\pmod k$. Since $\taubar_i$ starts the cycle, $\hat\tau_j < \hat\tau_i$ for all $j < i$. Also, if $i \neq n$ then $\hat\tau_i < \hat\tau_{i+1}$, otherwise then the cycle starting with $\taubar_i$ would contain all of the letters in block $i+1$, and thus would have at least $k + 1$ elements.

    Conversely, assume $\hat\tau_i$ is a fixed point of $f_2(\pi) = \sigma = \xtau$. Since $x \equiv 0 \pmod k$ and $\hat\tau_j < \hat\tau_i$ for all $j < i$, $\taubar_i$ starts a cycle whose length is a multiple of $k$. Since $\hat\tau_i < \hat\tau_{i+1}$, the cycle starting with $\taubar_i$ does not contain all of block $i+1$, and thus its length is $< 2k$, and so $\taubar_i$ starts a cycle of length $k$.
\end{proof}

\begin{proposition}\label{prop:k-cycle-fixed-pts}
    For any $\pi \in S_{kn}$, the number of $k$-cycles of $\pi$ equals the number of fixed points of $f_2(\pi)$.
\end{proposition}
\begin{proof}
    By \Cref{lem:fixed-points}, we have a bijection between the set of $k$-cycles in $\pi$ which begin with some $\taubar_i$, and the set of all fixed points of $f_2(\pi)$. By \Cref{lem:tauStartsKCycle}, all of the $k$-cycles of $\pi$ begin with some $\taubar_i$, so the result follows.
\end{proof}

\section{Inverting the bijection}
\label{sec:Bijection}
In this section, we construct a (right) inverse to $f$. Together with the fact that the domain and the codomain have the same cardinality, this shows that $f$ is a bijection.

Because the composition of maps $\pihat \mapsto \delta \mapsto \hat\delta$ involves creating $k$-cycles from $\pihat$ and putting this into canonical cycle structure, it has the result of permuting the blocks of $\hat\pi$ and cyclically shifting within the blocks. In order to invert this map, we define a cyclic shift on a block.
\begin{definition}
    For some $0 \leq s < k$, and $\block(\hat\pi,i) = \pihat_{k(i-1) + 1}\pihat_{k(i-1) + 2}\dots\pihat_{ki}$, we define define $\shift_s$ as the map that cyclically shifts the subword $s$ spaces to the left:
    \[
        \operatorname{shift}_s(\block(\hat\pi,i)) =
        \underbrace{
        \pihat_{k(i-1) + s + 1}
        \pihat_{k(i-1) + s + 2}
        \dots
        \pihat_{ki}
        }_{\text{length $k-s$ suffix}}
        \underbrace{
        \pihat_{k(i-1) + 1}
        \pihat_{k(i-1) + 2}
        \dots
        \pihat_{k(i-1) + s}
        }_{\text{length $s$ prefix}}.
    \]
\end{definition}

\begin{example}
    For $\block(\hat\pi,1) = 834$,
    \[
        \operatorname{shift}_0(834) = 834,\qquad
        \operatorname{shift}_1(834) = 348,\quad\text{and }\quad
        \operatorname{shift}_2(834) = 483.
    \]
\end{example}

The following theorem allows us to recover the order of the blocks in $\pihat$ from $\hat\delta$ and $\hat\tau$.
\begin{lemma}\label{lem:delta-shifts}
    Let $\pi\in S_{kn}$ and $f(\pi)=\deltaxtau$.
    Then for each $1 \leq i \leq n$ there exists a unique $s_i \in \Z_k$ such that
    \[
        \block(\hat\pi, i) = \shiftBlock{i}\!.
    \]
\end{lemma}
\begin{proof}
    Put an ordering on two blocks $b_1$ and $b_2$ where $b_1 > b_2$ if $\max b_1 > \max b_2$.
    By definition of $\hat\tau$,
    $\block(\pihat,i)$ is the $\hat\tau_i$-th smallest block of $\pihat$. Because $\hat\delta$ sorts the blocks of $\pihat$ in increasing order, the $i$-th block of $\hat\pi$ is the $\hat\tau_i$-th block of $\hat\delta$.
\end{proof}

The values of $s_i$ can be recovered from $\x$, which is constructed to make the fixed points of $\xtau$ correspond to the $k$-cycles of $\pi$.

\subsection{Example of inverting the map}
We use the output of Example \ref{ex:MainExample} to invert the map and recover $\pi$ from $\deltaxtau$.
\begin{example}
    Let \(\delta = \mathrm{(726)(834)(B59)(EDC)(F1A)}\), \(\x \equiv (0,1,0,2,1)\), and \(\tau = (2)(3)(514)\).
    We can immediately compute
    \(\hat\delta = \mathrm{726\,834\,B59\,EDC\,F1A}\),
    \(\hat\tau = \mathrm{23514}\), and
    \(\overline\tau = \mathrm{8BF7E}\).

    We permute the order of the blocks using $\tau$: label blocks $726,834,B59,EDC,F1A$ with $1,2,3,4,5$ respectively, then write them in the order $23514$ given by $\hat\tau$:
    \[726\,834\,B59\,EDC\,F1A \xrightarrow[\hat\tau]{} 834\,B59\,F1A\,726\,EDC\]
    Next, we use $\textbf{x}$ to determine cyclic shifts $s_1,\dots,s_5$ which will act on the blocks. We determine the shifts left to right; we must find $s_5$, then $s_4$, then $s_3$, and so on.

    \begin{description}
        \item[Determining $s_5$.]
        Because $x_{\hat\tau_5} = x_4 = 2$, we know that $\overline\tau_5 = E$ is a distance of $2 \pmod{3}$ from the end of its cycle, and thus $s_5 = 2$, and \(\block(\pihat, 5) = \shift_2(\mathrm{EDC}) = CED\).

        \item[Determining $s_4$.]
        Now, $\overline\tau_4 = 7$.
        We need to determine which letter starts and which letter ends its cycle. The largest letter weakly to the left of $\taubar_4 = 7$ will be $\overline\tau_3 = F$. This means that the cycle that contains $\taubar_4 = 7$ will start with $F$ and end with $D$ at the end of the word.
        Because $x_{\hat\tau_4} = x_1 = 0$, this means that $s_4 = 0$, and \(\block(\pihat, 4) = \shift_0(\mathrm{726})  = 726\).

        \item[Determining $s_3$.]
        Since $\overline\tau_3 = F$ is a record reading left-to-right, it starts its cycle, and its cycle ends with $D$. Thus because $x_{\hat\tau_3} = x_5 = 1$, we can see that $s_3 = 1$, and \(\block(\pihat, 3) = \shift_1(\mathrm{F1A}) = \mathrm{1AF}\).

        \item[Determining $s_2$.]
        Since $\overline\tau_2 = B$ is a record reading left-to-right, it starts its cycle, and its cycle ends with $A$, and since $x_3 = 0$, we can determine that $s_2 = 1$, and \(\block(\pihat, 2) = \shift_1(\mathrm{B59}) = \mathrm{59B}\).

        \item[Determining $s_1$.]
        Since $\overline\tau_1 = 8$ must be in a cycle that ends with $5$, and since $x_2 = 1$, we can determine that $s_1 = 0$, and \(\block(\pihat, 1) = \shift_0(\mathrm{834}) = \mathrm{834}\).
    \end{description}

    Letting $s_i$ act on the blocks yields:
    \[834\,B59\,F1A\,726\,EDC\xrightarrow[s_i]{} 834\,59B\,1AF\,726\,CED\]

    Using the inverse of Stanley's fundamental bijection, we recover: \[
        \boxed{\pi = \mathrm{(8345)(9)(B1A)(F726CED)}}.
    \]
    Note that the process used in this example is identical to the equality:
    \begin{align*}
        \hat\pi &= \shiftBlock{1} \cdots \shiftBlock{5} \\
                &= \shift_{s_1}(\mathrm{834})\shift_{s_2}(\mathrm{B59})\shift_{s_3}(\mathrm{F1A})\shift_{s_4}(\mathrm{726})\shift_{s_5}(\mathrm{EDC})
    \end{align*}
\end{example}

In the next section, we demonstrate how to calculate these shifts $s_1,\dots,s_n$.

\subsection{Shifting blocks}
Recall that we may write a permutation $\pi\in S_{kn}$ as a word of length $kn$ and consider this word in $n$ blocks, each of length $k$. We have previously defined $\overline{\tau}_i$ to be the largest letter in the $i$th block, which will sometimes be referred to as the $i$th block leader.

We introduce the following definition, where both $g_i$ and $d_i$ are concerned with the $i$th block leader, $\overline{\tau}_i$. $g_i$ represents the index after $\overline{\tau}_i$ where the next cycle begins, and $d_i$ gives the distance from $\overline{\tau}_i$ to this index.

\begin{definition}
    Write $\pi\in S_{kn}$ in canonical cycle notation. We call a letter of $\hat\pi$ a record if it begins a cycle of $\pi$, where $\pi$ is written in canonical cycle notation. For each $i\in \{1,\dots,n\}$, we define $g_i(\hat\pi)$ to be the index of the first record to the right of $\overline{\tau}_i$. If no such record exists, $g_i(\hat\pi) = kn+1$.

    Lastly, we define $d_i(\hat\pi)$ to be the difference between   $g_i(\hat\pi)$ and the index of $\overline{\tau}_i$.
\end{definition}

\begin{example}
    If $\hat\pi= 834\, 5 9 B\,1A F\,726\,CED$, then the records are $8$, $9$, $B$, $F$. The block leaders $\overline{\tau}_1,\dots,\overline{\tau}_5$ are $8$, $B$, $F$, $7$, $E$ respectively, and they have indices $1,6,9,10,14$. Then:
    \begin{align*}
        g_1(\pi) &= 5, d_1(\pi) = 4\\
        g_2(\pi) &= 9, d_2(\pi) = 3\\
        g_3(\pi) &= 16, d_3(\pi) = 7\\
        g_4(\pi) &= 16, d_4(\pi) = 6\\
        g_5(\pi) &= 16, d_5(\pi) = 2
    \end{align*}
    To be clear, $g_1(\pi) = 5$ because the next record in $\hat\pi$ after $\overline{\tau}_1 = 8$ is $9$, which has index $5$. $d_1(\pi) = 4$ because the index of $9$ is $5$ while the index of $8$ is $1$, so their difference is $d_1(\pi) = 5-1 = 4$.
\end{example}

We prove the following proposition, which asserts that the record following a block leader $\overline{\tau}_i$ cannot be in the same block as $\overline{\tau}_i$.

\begin{proposition}\label{prop:gi invariant}
    $g_i(\hat\pi) > ki$. In particular, $g_i(\hat\pi)$ is invariant up to cyclic shifts of $\block(\hat\pi, i)$.
\end{proposition}

\begin{proof}
    If there is no record to the right of $\overline{\tau}_i$, then $g_i(\hat\pi) = kn+1>ki$. Suppose there is a record to the right of $\overline{\tau}_i$. $\overline{\tau}_i$ is the largest letter in $\block(\hat\pi, i)$. As a result, other elements of $\block(\hat\pi, i)$ can never be the next record under any cyclic shift. It follows that the next record does not lie in $\block(\hat\pi, i)$, so $g_i(\hat\pi)>ki$. This record will be the same after any cyclic shift of $\block(\hat\pi, i)$, and the proposition is proven.
\end{proof}

\begin{lemma}\label{lem:computing-shifts}
    Let \[
        \hat w = \block(\hat\delta, \hat\tau_1)\cdots
        \block(\hat\delta, \hat\tau_i)
        \shiftBlock{i+1}
        \cdots
        \shiftBlock{n}
    \]
    and let $s_i \equiv x_{\hat\tau_i} -d_i(\hat\pi) \pmod k$. Then $s_i$ is the unique shift so that
    \[
        \block(\pihat,i) = \shift_{s_i}(\block(\hat\delta,\hat\tau_i)).
    \]
\end{lemma}
\begin{proof}
We prove inductively, starting with the base case $i = n$.
\begin{description}
    \item[Base case.]
    Assume $i = n$ and let \(\hat w = \block(\hat\delta,\hat\tau_1)\cdots\block(\hat\delta,\hat\tau_n)\). The index of $\overline{\tau}_n$ is $k(n-1) +1$ because each block of $\hat\delta$ begins with its block leader. Thus, $d_n(\hat w) \equiv kn+1-(k(n-1)+1) \equiv k\equiv 0 \pmod k$, $s_n \equiv x_{\hat\tau_n}\pmod k$, and thus $\shift_{s_n}\!\!\left(\block(\hat\delta,\hat\tau_n)\right)$ is the unique shift that places $\overline\tau_n$ at a distance of $x_{\hat\tau_n}$ from the end of the word. Thus by by \Cref{lem:delta-shifts}, $\block(\pihat,n) = \shift_{s_n}(\block(\hat\delta,\hat\tau_n))$.

    \item[Inductive step.]
    Assume that $\block(\pihat,j) = \shiftBlock{j}$ for all $j$ such that $i + 1 \le j \leq n$, and let \[
        \hat w = \block(\hat\delta,\hat\tau_1)\cdots\block(\hat\delta,\hat\tau_i)\block(\pihat,i+1) \cdots \block(\pihat, n).
    \] By \Cref{prop:gi invariant}, $g_i(\hat w)$ is invariant up to cyclically shifting the $i$-th block, so $g_i(\hat w) = g_i(\pihat)$. Before shifting, the distance from $\overline{\tau}_i$ to $g_i(\hat w)$ is $d_i(\hat w)$. Thus, $s_i = x_{\hat\tau_i}-d_i(\hat w)$ is the unique shift of the $i$-th block so that $\taubar_i$ is a distance of $x_{\hat\tau_i} \pmod k$ from $g_i(\hat w)$. As a result, $\shiftBlock{i} = \block(\pihat,i)$ by \Cref{lem:delta-shifts}.
\end{description}
\end{proof}

\begin{definition}
    We define $f^{-1}\!\deltaxtau = \pi'$, the inverse of Stanley's fundamental bijection applied word
    \[
        \hat\pi' = \shift_{s_1}(\block(\hat \delta,\hat\tau_1))
        \cdots
        \shift_{s_n}(\block(\hat \delta,\hat\tau_n)),
    \]
    where each shift $s_i\in\mathbb{Z}_k$ is as defined above.
\end{definition}

\begin{theorem}
    The map $f \colon S_{kn} \to D_{k,n} \times S(k,n)$ is a bijection.
\end{theorem}
\begin{proof}
    By \Cref{lem:computing-shifts}, we can see that ${\hat\pi}' = \pihat$, and thus $f^{-1} \circ f(\pi) = \pi$, and so $f^{-1}$ is a left-inverse of $f$ and thus $f$ is injective. Because $|D_{k,n}| = (kn)!/(k^nn!)$ and $|S(k,n)| = k^nn!$, the domain and the codomain have the same cardinality: $|S_{kn}| = (kn)! = |D_{k,n} \times S(k,n)|$. Therefore $f$ is bijective and $f^{-1}$ is a two-sided inverse of $f$.
\end{proof}

We have found the desired bijection $f$, and by \Cref{thm:involution}, this induces the desired statistic-swapping involution $\phi: S(k,n)\times S_{kn}\to S(k,n)\times S_{kn}$.

\subsection{The involution}
Recall that \Cref{thm:involution} gives a construction of a statistic-swapping involution $\phi\colon S(k,n) \times S_{kn} \to S(k,n) \times S_{kn}$ predicated on a statistic-preserving bijection $f\colon S_{kn} \to D_{k,n} \times S(k,n)$. Here we compute an example.
\begin{example}
    Continuing from \Cref{ex:MainExample}, again let $k = 3$, $n=5$, and \(\pi = \mathrm{(8345)(9)(B1A)(F726CED)} \in S_{kn}\). Moreover, suppose that \[
        \sigma' = \big(\underbrace{(2,0,0,1,0)}_{\x},\underbrace{(2)(31)(4)(5)}_\tau\big).
    \] We will compute $\phi(\pi,\sigma')$. Note that $\pi$ has $1$ $k$-cycle and $\sigma'$ has $2$ fixed points (i.e. indices $2$ and $5$), so $\phi(\pi,\sigma)$ should return a permutation with $2$ $k$-cycles and a generalized symmetric group element with $1$ fixed point. Recall that \(f(\pi) = (\delta,\sigma)\) where \begin{align*}
        \delta &= \mathrm{(726)(834)(B59)(EDC)(F1A)} \\
        \sigma &= \xtau = \left((0,1,0,2,1),(2)(3)(514)\right),
    \end{align*}
    and $\sigma$ has one fixed point (i.e. index $3$.)
    Following the construction in \Cref{thm:involution}, we compute \[
        f^{-1}(\delta, \sigma') = \pi'= \mathrm{(834)(B59672)(DC)(E)(F1A)},
    \] which has $2$ $k$-cycles, as desired.

    Thus \(\phi(\pi,\sigma') = (\pi',\sigma)\) is an involution which swaps the desired statistics: $\pi$ and $\sigma$ have one $k$-cycle and fixed-point respectively, and $\pi'$ and $\sigma'$ have two $k$-cycles and fixed-points respectively.
\end{example}

It is worth briefly noting that when $k \geq 2$, $\phi$ is not a group isomorphism. This can be seen by considering the identity in $S_{kn} \times S(k,n)$ which has no $k$-cycles and $n$ fixed points, and thus the statistic-swapping involution cannot send the identity to itself. Indeed, the bijection provided here is purely combinatorial, and the construction ignores the group structure.

\section{Further considerations}

To conclude the paper, we briefly mention a paper of Duchon and Duvignau \cite{DuchonDuvignau2016}. They construct a reversible insertion algorithm $\mathcal{A}\colon [kn-j+1]^{(j)} \times S_{kn-j} \to S_{kn}$ which takes a word in the set $[kn-j+1]^{(j)} = [kn-j+1] \times [kn-j+2] \times \dots \times [kn]$ and a permutation in $S_{kn-j}$ and outputs a permutation in $S_{kn}$ such that the two permutations have the same number of $k$-cycles whenever $j < k$. This, together with our work, gives a combinatorial proof that for all $j < k$, \[
    \frac{
        \left|\Cyc_m(S_{kn-j})\right|
    }{
        \left|S_{kn-j}\right|
    } =
    \frac{
        \left|\Fxpt_m\!\left(S(k,n)\right)\right|
    }{
        \left|S(k,n)\right|
    },
\]

This induces a statistic-swapping involution $\psi\colon [kn-j+1]^{(j)} \times S_{kn-j} \times S(k,n)$. Define $\psi(\ell, \pi, \sigma') = (\ell', \pi', \sigma)$, where \begin{align*}
    (\pi'', \sigma) &= \phi(\mathcal A(\ell,\pi), \sigma') \quad\text{and}\\
    (\ell', \pi') &= \mathcal A^{-1}(\pi'')
\end{align*} Here $\pi'$ and $\pi''$ each have the same number of $k$-cycles as $\sigma'$ has fixed points. Meanwhile, $\pi$ has the same number of $k$-cycles as $\sigma$ has fixed points.

\section*{Acknowledgments}
We thank Sam Armon for his useful feedback and suggestions; Sami Assaf for providing the original motivation; and David Kempe, Richard Arratia, J.E. Paguyo, and Michael E. Orrison for their comments.
\bibliographystyle{plain}
\bibliography{references.bib}
\end{document}